\documentclass[11pt]{amsart}
 
\usepackage{amsmath}
\usepackage{amsfonts,amssymb,mathrsfs,amscd}

\numberwithin{equation}{section}
\theoremstyle{plain}
\newtheorem{theorem}{Theorem}[section]
\newtheorem{lemma}{Lemma}[section]
\newtheorem{propos}{Proposition}[section]
\theoremstyle{definition}
\newtheorem{definition}{Definition}
\newtheorem{remark}{Remark}[section]

\begin{document}

\title{Laplacians on smooth distributions}
\author[Yuri A.~Kordyukov]{Yuri A.~Kordyukov}
\address{Institute of mathematics, Ufa Scientific Center, Russian Academy of Sciences}
\email{yurikor@matem.anrb.ru}


\begin{abstract}
Let $M$ be a compact smooth manifold equipped with a positive smooth
density $\mu$ and $H$ be a smooth distribution endowed with a
fiberwise inner product $g$. We define the Laplacian $\Delta_H$
associated with $(H,\mu,g)$ and prove that it gives rise to an unbounded
self-adjoint operator in $L^2(M,\mu)$. Then, assuming that $H$ generates a singular foliation $\mathcal F$, we prove that, for
any function $\varphi$ from the Schwartz space $\mathcal S(\mathbb R)$, the operator $\varphi(\Delta_H)$ is a smoothing operator in the scale of
longitudinal Sobolev spaces associated with $\mathcal F$. The proofs are based on pseudodifferential calculus on singular foliations developed by Androulidakis and Skandalis and subelliptic estimates for $\Delta_H$.
\end{abstract}

\footnotetext[0]{Supported by the Russian Foundation of Basic Research  (grant~16-01-00312).}
\maketitle


\section{Introduction}\label{s:intro}
The main purpose of the paper is to define and study some natural
geometric differential operators associated with an arbitrary smooth
distribution on a compact manifold.
Let $M$ be a connected compact smooth manifold of dimension $n$
equipped with a positive smooth density $\mu$. Let $H$ be a smooth
rank $p$ distribution on $M$ (that is, $H$ is a smooth subbundle of
the tangent bundle $TM$ of $M$) and $g$ be a smooth fiberwise inner
product on $H$. We define the horizontal differential $d_Hf$ of a function
$f\in C^\infty(M)$ to be the restriction of its differential $df$ to
$H\subset TM$. Thus, $d_Hf$ is a section of the dual bundle $H^*$ of
$H$, $d_Hf\in C^\infty(M,H^*)$, and we get a first order
differential operator $d_H : C^\infty(M) \to C^\infty(M,H^*)$. The
Riemannian metric $g$ and the positive smooth density $\mu$ induce
inner products in $C^\infty(M)$ and $C^\infty(M,H^*)$, that allows
us to consider the adjoint $d^*_H : C^\infty(M,H^*)\to C^\infty(M)$
of $d_H$. Finally, the Laplacian $\Delta_H$ associated with
$(H,g,\mu)$ is the second order differential operator on
$C^\infty(M)$ given by
\[
\Delta_H=d^*_Hd_H.
\]
If $X_k, k=1,\ldots,p,$ is a local orthonormal frame in $H$ defined
on an open subset $\Omega\subset M$, then one can easily check that the
restriction of $\Delta_H$ to $\Omega$ is given by
\[
\Delta_H\left|_\Omega\right.=\sum_{k=1}^p X^*_kX_k.
\]
The next theorem allows us to talk about spectral properties of the operator $\Delta_H$.

\begin{theorem}\label{mainthm}
The Laplacian $\Delta_H$ considered as an unbounded
operator in the Hilbert space $L^2(M,\mu)$ with domain
$C^\infty(M)$ is essentially self-adjoint.
\end{theorem}

One can give a proof of Theorem~\ref{mainthm}, using a
well-known result of Chernoff \cite{Chernoff} based on the theory of
first order linear symmetric hyperbolic systems. This proof is given
in Section~\ref{s:proofs}. We also present another proof of Theorem~\ref{mainthm}, which is more complicated, but we hope that the techniques used in this proof will be helpful for the study of more refined spectral properties of the operator $\Delta_H$.

If the distribution $H$ is completely integrable, then, by the
Frobenius theorem, it gives rise to a smooth foliation $\mathcal F$
on $M$. In this case, the operator $\Delta_H$ is a formally self-adjoint longitudinally elliptic operator with respect to $\mathcal F$. Spectral properties of this operator, in particular, its self-adjointness have been studied in several papers (see, for instance, \cite{Connes79,Kord95,Vassout} and the references therein). Here an important role is played by the longitudinal pseudodifferential calculus for foliations developed by Connes in \cite{Connes79}. On the other
hand, if $H$ is completely nonintegrable (or bracket-generating), then,
using H{\"o}rmander's sum of the squares theorem \cite{Hormander67}, one can show that the operator $\Delta_H$ is hypoelliptic, that easily implies its self-adjointness. The proof of Theorem~\ref{mainthm} in the general
case combines two approaches mentioned above. We assume that the distribution $H$ defines a singular foliation $\mathcal F$ in the sense of Stefan and Sussmann
\cite{Stefan,Sussmann}. Then the operator $\Delta_H$ can be considered as a longitudinally hypoelliptic operator with respect to $\mathcal F$. In \cite{Andr-Skandalis-II}, Androulidakis and Skandalis developed a pseudodifferential calculus on singular foliations. Following Kohn's proof of H{\"o}rmander's sum of the squares theorem, we derive subelliptic estimates and prove longitudinal hypoellipticity for the operator $\Delta_H$ in the scale of longitudinal Sobolev spaces on $M$ associated with the singular foliation $\mathcal F$. Using these results, we easily complete the proof of Theorem~\ref{mainthm}.

Theorem~\ref{mainthm} allows us to consider more refined spectral
properties of the Laplacian $\Delta_H$. First of all, by spectral
theorem, we can consider functions of $\Delta_H$ such as the heat
operator $e^{-t\Delta_H}$, the wave operator $e^{it\sqrt{\Delta_H}}$
and so on. Using the longitudinal hypoellipticity result mentioned above, we immediately get the following theorem.

\begin{theorem}\label{t:smoothing}
Suppose that the distribution $H$ defines a singular foliation $\mathcal F$. 
For any function $\varphi$ from the Schwartz space $\mathcal
S(\mathbb R)$, the operator $\varphi(\Delta_H)$ extends to a bounded
operator from $H^s(\mathcal F)$ to $H^t(\mathcal F)$ for any $s,t\in
\mathbb R$.
\end{theorem}

We can also use the spectral properties of the operator $\Delta_H$
to define invariants of smooth distributions. For instance, one
can consider the class of distributions $H$ such that the spectrum
of the associated Laplacian $\Delta_H$ has a gap near zero. It is
easy to see that this property of $\Delta_H$ is independent of the
choice of $g$ and $\mu$. For smooth foliations, it is apparently related with property (T) for its holonomy groupoid (see, for instance, a discussion in \cite[Remark 10]{mathann}). To
study more refined invariants of distributions, it would be nice to
have some natural way to choose $\mu$ and $g$ that would
give rise to the intrinsic Laplacian associated with $H$. The
problem of the Laplacian and of the intrinsic Laplacian was
extensively discussed recently in sub-Riemannian geometry (see, for
instance, \cite{Montgomery-book,Agrachev-BGR09,Gordina-Laetsch,HK}
and references therein). In the general case, such an intrinsic
choice is not always possible. For instance, in the case when $H$ is
integrable, $g$ and $\mu$ look completely independent: $g$ is
responsible for the longitudinal structure and $\mu$ for the
transverse one.

In \cite{Omori-Kobayashi,Petronilho,Shimoda} (see also the
references therein), the authors studied global hypoellipticity of
H{\"o}rmander's sum of the squares operators. In the case when $H$ has
transverse symmetries given by a Riemannian foliation, orthogonal to
$H$, the associated Laplacian (sometimes called the horizontal
Laplacian) was studied in
\cite{Baudoin14,Baudoin-Kim,noncom,JGP07,AGAG08,Prokhorenkov-Richardson}
(see also the references therein). In particular, its
self-adjointness was established in \cite{noncom}. In \cite{GGK}
(see also \cite{Daniel-Ma}), the authors introduced the
characteristic Laplacian associated with an arbitrary smooth
distribution $H$ and a Riemannian metric on $M$. This operator
coincides with the operator $\Delta_H$ in degree $0$, if $g$ is the
restriction of the Riemannian metric to $H$ and $\mu$ is the
Riemannian volume form. The problem of constructing natural
geometric operators on differential forms associated with an
arbitrary smooth distribution is a very interesting open problem
(see, for instance, \cite{Nicolaescu,Petit,Rumin} and references
therein for some related results in sub-Riemannian geometry).

The paper is organized as follows. In Section~\ref{s:proofs}, we
state theorems on subelliptic estimates and longitudinal hypoellipticity for the Laplacian $\Delta_H$ and show how these results enable us to prove the main results of the paper. In Section~\ref{s:foliations}, we give
necessary information about singular foliations and
pseudodifferential calculus on singular foliations.
Section~\ref{s:subelliptic} contains the proofs of the theorems on subelliptic
estimates and longitudinal hypoellipticity stated in Section~\ref{s:proofs}.

The author is grateful to I. Androulidakis and G. Skandalis for very useful discussions and remarks and to the anonymous referee for suggestions to improve the paper.

\section{Longitudinal hypoellipticity and proofs of main results}\label{s:proofs}

As above, let $M$ be a connected compact smooth manifold of dimension $n$ equipped with a positive smooth density $\mu$. Let $H$ be a smooth
rank $p$ distribution on $M$ and $g$ be a smooth fiberwise inner
product on $H$. Consider the $C^\infty(M)$-module $C^\infty(M,TM)$ of smooth vector fields on $M$. It is a Lie algebra with respect to the Lie bracket. Denote by $C^\infty(M,H)$ the submodule of $C^\infty(M,TM)$, which consists of smooth vector fields, tangent to $H$ at each point.  Let $\mathcal F$ be the minimal submodule of $C^\infty(M,TM)$, which contains $C^\infty(M,H)$ and is stable under Lie brackets. We assume that $\mathcal F$ is locally finitely generated. Then it is a singular foliation in the sense of Stefan and Sussmann. We will use classes $\Psi^m({\mathcal F})$ of longitudinal pseudodifferential operators and the corresponding scale
$H^s(\mathcal F)$ of longitudinal Sobolev space associated with
$\mathcal F$ (we refer the reader to Section~\ref{s:foliations} for necessary
information about singular foliations and pseudodifferential
calculus on singular foliations).

First, we state subelliptic estimates for the operator $\Delta_H$. 

\begin{theorem}\label{t:Hs-hypo}
There exists $\epsilon>0$ such that, for any $s\in \mathbb R$, we have
\[
\|u\|_{s+\epsilon}^2 \leq
C_s\left(\|\Delta_Hu\|_s^2+\|u\|_s^2\right), \quad u\in C^\infty(M),
\]
where $C_s>0$ is some constant and $\|\cdot\|_s$ denotes the norm in $H^s(\mathcal F)$.
\end{theorem}

As a consequence, we get the following longitudinal hypoellipticity result. 

\begin{theorem}\label{t:hypo}
If $u\in H^{-\infty}(\mathcal F):=\bigcup_{t\in \mathbb R}H^t(\mathcal F)$ such that $\Delta_Hu\in H^s(\mathcal F)$ for some $s\in \mathbb R$, then $u \in H^{s+\varepsilon}(\mathcal F)$. 
\end{theorem}

The proofs of Theorems~\ref{t:Hs-hypo} and \ref{t:hypo} will be given in
Section~\ref{s:subelliptic}. Here we show how to prove
Theorems~\ref{mainthm} and~\ref{t:smoothing} on the base of these
theorems.

\begin{proof}[Proof of Theorem~\ref{mainthm}]
By the basic criterion of essential self-ad\-joint\-ness, it is
sufficient to show that $\ker(\Delta_H^*\pm i)=\{0\}$, where
$\Delta_H^*$ is the adjoint of $\Delta_H$ considered as an unbounded
linear operator in $L^2(M,\mu)$ with domain $C^\infty(M)$. Moreover,
it is sufficient to show that $\ker(\Delta_H^*\pm i)$ is
contained in the domain ${\rm Dom}\,\overline{\Delta_H}$ of the closure of $\Delta_H$ in $L^2(M,\mu)$. Let $u\in \ker(\Delta_H^*\pm i)$. So we have $u\in
L^2(M,\mu)$ and $(\Delta_H^*\pm i)u=0$. Since $\Delta_H$ is
symmetric on $C^\infty(M)$, we obtain that $(\Delta_H\pm i)u=0$,
where $\Delta_Hu$ is understood in the distributional sense. Taking into account that $u\in L^2(M,\mu)\subset H^{-\infty}(\mathcal
F)$ and using Theorem~\ref{t:hypo}, we obtain that $u$ is in $H^\infty(\mathcal
F):=\bigcap_{t\in \mathbb R}H^t(\mathcal F)$. This immediately completes the proof, because it is easy to see that 
$H^2(\mathcal F)$ is contained in ${\rm Dom}\,\overline{\Delta_H}$ (see
Theorem~\ref{t:action_in_Sobolev} and Proposition~\ref{p:density_in_Sobolev} below).
\end{proof}

\begin{remark}
Observe that the paper \cite{Andr-Skandalis-II} deals with Hilbert
modules over the $C^*$-algebra $C^*(M,\mathcal F)$ of the singular
foliation $\mathcal F$. Unlike \cite{Andr-Skandalis-II}, we
work not with Hilbert modules and $C^*$-algebras, but with the concrete
representation of the $C^*$-algebra $C^*(M,\mathcal F)$ on $L^2(M)$.
This enables us to use some results of theory of linear operators in Hilbert
spaces (first of all, the basic criterion of essential self-ad\-joint\-ness).
It would be very interesting to prove a $C^*$-module version of Theorem~\ref{mainthm}, stating that $\Delta_H$ gives rise to a regular
(unbounded) self-adjoint multiplier of $C^*(M,\mathcal F)$. For
longitudinally elliptic operators on $\mathcal F$, this was proved
in \cite{Andr-Skandalis-II}, extending a similar result for regular
foliation by Vassout \cite{Vassout} (see also \cite{Kord95}).
We also note that the proof of Theorem~\ref{t:Hs-hypo} can be easily extended to the Hilbert $C^*$-module setting.
\end{remark}

\begin{proof}[Proof of Theorem~\ref{t:smoothing}]
Let $\varphi\in \mathcal S(\mathbb R)$. Then, by Theorem
\ref{t:L2-hypo}, we have
\[
\|\varphi(\Delta_H)u\|_{\epsilon}^2 \leq
C\left(\|\Delta_H\varphi(\Delta_H)u\|^2+\|\varphi(\Delta_H)u\|^2\right)\leq
C_1(\varphi)\|u\|^2, u\in C^\infty(M).
\]
Therefore, the operator $\varphi(\Delta_H)$ defines an operator from $L^2(M,\mu)$ to $H^\epsilon(\mathcal F)$.
Repeating this argument, we obtain that, for any $\varphi\in
\mathcal S(\mathbb R)$, the operator $\varphi(\Delta_H)$ defines an
operator from $L^2(M,\mu)$ to $H^s(\mathcal F)$ and, by duality, from $H^{-s}(\mathcal F)$ to
$L^2(M,\mu)$ for any $s\geq 0$. It remains to show that, for any
$\varphi\in \mathcal S(\mathbb R)$, the operator $\varphi(\Delta_H)$
defines an operator from $H^{-t}(\mathcal F)$ to $H^s(\mathcal F)$ for any $s,t\geq 0$.

The operator $\Delta_H+1$ is invertible in $L^2(M,\mu)$, and, by
Theorem \ref{t:L2-hypo}, the inverse
$(\Delta_H+1)^{-1}$ acts from $L^2(M,\mu)$ to $H^\epsilon(\mathcal F)$:
\[
\|(\Delta_H+1)^{-1}u\|_{\epsilon}\leq C\|u\|, \quad v\in
C^\infty(M).
\]
Using Theorem \ref{t:Hs-hypo} repeatedly, we obtain that, for any natural $N$, the operator $(\Delta_H+1)^{-N}$ acts from $L^2(M,\mu)$ to $H^{N\epsilon}(\mathcal F)$ and,
by duality, from $H^{-N\epsilon}(\mathcal F)$ to $L^2(M,\mu)$:
\[
\|(\Delta_H+1)^{-N}u\| \leq C\|u\|_{-N\epsilon}, \quad u\in
C^\infty(M).
\]
Finally, for any $\varphi\in \mathcal S(\mathbb R)$, $s>0$ and
natural $N$, we get
\begin{multline*}
\|\varphi(\Delta_H)u\|_s=
\|\varphi(\Delta_H)(\Delta_H+1)^{N}(\Delta_H+1)^{-N}u\|_s \\ \leq C\|(\Delta_H+1)^{-N}u\| \leq C\|u\|_{-N\epsilon}, \quad u\in C^\infty(M).
\end{multline*}
Thus, the operator $\varphi(\Delta_H)$ defines an operator from
$H^{-N\epsilon}(\mathcal F)$ to $H^s(\mathcal F)$.
\end{proof}

At the end of this section, we recall the proof of
Theorem~\ref{mainthm} mentioned in Introduction, which is based on
the theory of first order linear symmetric hyperbolic systems. Here we don't assume that the distribution $H$ defines a singular foliation. 

\begin{proof}[Proof of Theorem~\ref{mainthm}]
On the Hilbert space $H=L^2(M,\mu)\oplus L^2(M,H^*,\mu)$, consider the operator $A$, with domain $D(A)=C^\infty(M)\oplus
C^\infty(M,H^*)$, given by the matrix
\[
A=\begin{pmatrix} 0 & d^*_H\\
d_H & 0
\end{pmatrix}.
\]
Observe that the operator $A$ is symmetric. Applying \cite[Theorem
2.2]{Chernoff} to the skew-symmetric operator $L=iA$, we obtain that
every power of $A$ is essentially self-adjoint. Since
\[
A^2=\begin{pmatrix} d^*_Hd_H & 0\\
0 & d_Hd^*_H
\end{pmatrix},
\]
the operator $d^*_Hd_H$ is essentially self-adjoint on
$C^\infty(M)$.
\end{proof}

\section{Preliminaries}\label{s:foliations}

In this section, we will give necessary information about singular
foliations and describe basic facts of pseudodifferential calculus
on singular foliations, mostly due to
\cite{Andr-Skandalis-I,Andr-Skandalis-II}, adapted to a concrete
representation in the $L^2$ space on the ambient manifold $M$.

\subsection{Foliations and bi-submersions}

Let $M$ be a smooth manifold. Consider the $C^\infty_c(M)$-module
$C^\infty_c(M, TM)$ of smooth, compactly supported vector fields on
$M$. As in \cite{Andr-Skandalis-I}, by a singular foliation $\mathcal F$ on
$M$, we will mean a locally finitely generated
$C^\infty_c(M)$-submodule of $C^\infty_c(M, TM)$ stable under Lie
brackets. Here a submodule $\mathcal E$ of $C^\infty_c(M, TM)$ is
said to be locally finitely generated if, for any $p\in M$, there
exists an open neighborhood $U$ of $p$ in $M$ and vector fields
$X_1,\ldots, X_k\in C^\infty_c(U,TU)$ such that, for any $f\in
C^\infty_c(U)$ and $X\in C^\infty_c(M,TM)$, we have
$fX\left|_U\right.=\sum_{j=1}^kf_jX_j\in C^\infty_c(U,TU)$ with some
$f_1,\ldots, f_k\in C^\infty_c(M)$.

Let $\mathcal F$ be a foliation on $M$ and $x\in M$. The tangent
space of the leaf is the image $F_x$ of $\mathcal F$ in $T_xM$ under
the evaluation map $C^\infty_c(M, TM) \to T_xM, x\mapsto X(x)$. Put
$I_x=\{f\in C^\infty(M): f(x)=0\}$. The fiber of $\mathcal F$ at $x$
is the quotient $\mathcal F_x=\mathcal F/I_x\mathcal F$. The
evaluation map induces a short exact sequence of vector spaces
\[
0\longrightarrow \mathfrak g_x \longrightarrow \mathcal F_x
\longrightarrow F_x \longrightarrow 0,
\]
where $\mathfrak g_x$ is a Lie algebra. One can show that, if, for
$x\in M$, the images of $X_1, \ldots , X_n \in \mathcal F$ in
$\mathcal F_x$ form a base of $\mathcal F_x$, then there exists a
neighborhood $U$ of $x$ in $M$ such that $\mathcal F$ restricted to
$U$ is generated by the restrictions of $X_1, \ldots , X_n$ to $U$.

For any smooth map $p:N\to M$ of a smooth manifold $N$ to $M$, we
denote by $p^{-1}(\mathcal F)$ the set of all vector fields on $N$
of the form $fY\in C^\infty_c(N;TN)$, where $f\in C^\infty_c(N)$ and
$Y$ is a vector field on $N$, which is $p$-related with some $X\in
\mathcal F$:  $dp_x(Y(x)) =X(p(x))$ for any $x\in N$. One can show
that $p^{-1}(\mathcal F)$ is a submodule of $C^\infty_c(N;TN)$.

\begin{definition}\label{d:bisub}
A bi-submersion of $(M,\mathcal F)$ is a smooth manifold $U$ endowed
with two smooth maps $s,t:U\to M$ which are submersions and satisfy:
\begin{description}
\item[(a)] $s^{-1}(\mathcal F)=t^{-1}(\mathcal F)$;
\item[(b)] $s^{-1}(\mathcal F)=C^\infty_c(U;\ker ds)+C^\infty_c(U;\ker dt)$.
\end{description}
\end{definition}

\begin{definition}
A locally closed submanifold $V\subset U$ is said to be an identity
bisection of a bi-submersion $(U,t,s)$ if the restriction
$s\left|_V\right. : V\to M$ (resp. $t\left|_V\right. : V\to M$) of
$s$ (resp. $t$) to $V$ is a diffeomorphism to an open subset $s(V)$
(resp. $t(V)$) of $M$, and, moreover,
$s\left|_V\right.=t\left|_V\right. $.
\end{definition}

An important class of bi-submersions is constructed in
\cite[Proposition 2.10(a)]{Andr-Skandalis-I}. Let $x\in M$. Let
$X_1,\ldots,X_n\in \mathcal F$ be vector fields whose images in
$\mathcal F_x$ form a basis of $\mathcal F_x$. For
$y=(y_1,\ldots,y_n)\in \mathbb R^n$, put $\varphi_y=\exp(\sum
y_iX_i)\in \exp\mathcal F$. Put $\mathcal W_0=\mathbb R^n\times M$,
$s_0(y,x)=x$ and $t_0(y,x)=\varphi_y(x)$. One can show that there is
a neighborhood $\mathcal W$ of $(0,x)$ in $\mathcal W_0$ such that
$(\mathcal W,t,s)$ is a bi-submersion where $s=s_0\left|_{\mathcal
W}\right.$ and $t=t_0\left|_{\mathcal W}\right.$. Such a
bi-submersion is called an identity bi-submersion.

A simple way to produce more bi-submersions, starting from the given
one, is described in \cite[Lemma 2.3]{Andr-Skandalis-I}. If
$(U,t,s)$ is a bi-submersion and $p:W\to U$ is a submersion, then
$(W,t\circ p,s\circ p)$ is a bi-submersion.

A morphism of bi-submersions $(U_i,t_i,s_i)$, $i=1,2$, is a smooth
map $f: U_1\to U_2$ such that, for any $u\in U_1$,
$s_1(u)=s_2(f(u))$ and $t_1(u)=t_2(f(u))$. Any submersion $p:W\to U$ is a morphism of bi-submersions $(W,t\circ p,s\circ p)$ and $(U,t,s)$.

As shown in \cite[Proposition 2.10(a)]{Andr-Skandalis-I}, the
identity bi-submersion $W$ at $x\in M$ provides a local model for
any bi-submersion, admitting a non-empty identity bisection,
containing $x$. More precisely, let $(V,t_V,s_V)$ be a bi-submersion
and $W\subset V$ be an identity bisection. Then, for any $v\in W$
with $s_V(v)=x\in M$, there exist an open neighborhood $V^\prime$
of $v$ in $V$ and a submersion $g:V^\prime \to \mathcal W$ which is
a morphism of bi-submersions, such that $g(v)=(0,x)$.

A stronger statement is shown in \cite[Lemma
2.5]{Andr-Skandalis-II}. Let $(U_j,t_j,s_j)$, $j=1,2$, be
bi-submersions, $V_j\subset U_j$ identity bisections and $u_j\in
V_j$ such that $s_1(u_1)=s_2(u_2)$. Then there exist an open
neighborhood $U^\prime_1$ of $u_1$ in $U_1$ and a morphism of
bi-submersions $g:U^\prime_1 \to U_2$ such that $g(u_1)=u_2$ and
$g(V_1\cap U^\prime_1)\subset U_2$.

For any bi-submersions $(U_j,t_j,s_j)$, $j=1,2$, we define their
composition $(U_1,t_1,s_1)\circ (U_2,t_2,s_2)=(U_1\circ U_2,t,s)$ as
follows. The manifold $U_1\circ U_2$ is the fiber product
\[
U_1\circ U_2=U_1\times_M U_2=\{(u_1,u_2)\in U_1\times U_2;
s_1(u_1)=t_2(u_2)\},
\]
and $s(u_1,u_2)=s_2(u_2)$ and $t(u_1,u_2)=t_1(u_1)$. For a
bi-submersion $(U,t,s)$, define its inverse as
$(U,t,s)^{-1}=(U,s,t)$. One can show that $U_1\circ U_2$ and
$U^{-1}$ are bi-submersions.

Denote by $\mathcal U_0$ the set of bi-submersions generated by
identity bi-submersions, that is, the minimal set of bi-sub\-mer\-sions
that contains all the identity bi-submersions and is closed under
operations of composition and taking inverse. $\mathcal U_0$ is
called the path holonomy atlas.

Sometimes, it is useful to extend the class of bi-submersions under considerations. We say that a bi-submersion $(W,t_W,s_W)$ is adapted to $\mathcal U_0$ at $w\in W$ if there exists an open subset $W^\prime\subset W$ containing $v$, a bi-submersion $(U,t,s)\in \mathcal U_0$ and a
morphism of bi-submersion $W^\prime\to U$. A bi-submersion
$(W,t_W,s_W)$ is adapted to $\mathcal U_0$ if for all $w\in W$,
$(W,t_W,s_W)$ is adapted to $\mathcal U_0$ at $w$.

\subsection{Regularizing operators}
From now on, we will assume that $M$ is compact.
In this section, we recall the definition of regularizing (or
leafwise smoothing) operators on $M$.  Our constructions will be adapted to a certain Hilbert structure in $L^2(M)$. Actually, we will describe a
$\ast$-representation in $L^2(M)$ of some involutive operator algebra associated with $\mathcal F$, which was introduced in \cite{Andr-Skandalis-I}. First, we fix a positive smooth density $\mu$ on $M$. For a vector space $E$ and $p\in (0,1]$, denote by $\Omega^pE$ the space of $p$-densities on $E$. 

Suppose that $(U,t,s)$ is a bi-submersion. Denote by $\Omega^{1/2}U$ the half-density bundle associated with the bundle $\ker ds \oplus \ker dt$ on $U$:
\[
\Omega^{1/2}U=\Omega^{1/2}\ker ds\otimes \Omega^{1/2}\ker dt.
\]
As shown in \cite[Section 3.2.1]{Andr-Skandalis-II}, for any quasi-invariant measure $\mu$ on $M$, there exists a measurable almost everywhere invertible section $\rho^U$ of $\Omega^{-1/2}\ker ds\otimes \Omega^{1/2}\ker dt$ on $U$ such that for every $f\in C_c(U; \Omega^{1/2}U)$ we have
\[
\int_M\left(\int_{s^{-1}(x)}(\rho^U_u)^{-1}\cdot f(u)\right)d\mu(x)=\int_M\left(\int_{t^{-1}(x)}\rho^U_u\cdot f(u)\right)d\mu(x).
\]
Here $(\rho^U)^{-1}\cdot f$ is a measurable section of $
\Omega^{1}\ker ds$ on $U$, which can be integrated along the fibers of $s$, giving rise to a function on $M$, and $\rho^U\cdot f$ is a measurable section of $\Omega^{1}\ker dt)$ on $U$, which can be integrated along the fibers of $t$.

If $\mu$ is given by a smooth positive density on $M$, $\rho^U$ can be constructed in the following way, which also shows its smoothness. First, for $u\in U$, we observe a short exact sequence
\[
0\longrightarrow \ker ds_u\longrightarrow T_uU\stackrel{ds_u}{\longrightarrow} TM_{s(u)}\longrightarrow 0, 
\]
which gives rise to an isomorphism
\[
\Omega^{1/2}T_uU\cong \Omega^{1/2}\ker ds_u\otimes \Omega^{1/2}TM_{s(u)}.
\]
Similarly, we get an isomorphism
\[
\Omega^{1/2}T_uU\cong \Omega^{1/2}\ker dt_u\otimes \Omega^{1/2} TM_{t(u)}.
\]
The smooth positive density $\mu$ on $M$ defines isomorphisms $
\Omega^{1/2}TM_{s(u)}\cong \mathbb C$ and $\Omega^{1/2}TM_{t(u)}\cong \mathbb C$. Combining these isomorphisms, we obtain a smooth invertible section $\rho^U$ of the bundle $\Omega^{-1/2}\ker ds\otimes \Omega^{1/2}\ker dt$.

\begin{definition}
For a bi-submersion $(U,t_U,s_U)$, the regularizing operator $R_U(k) : L^2(M)\to L^2(M)$ associated with the longitudinal kernel $k\in C^\infty_c(U, \Omega^{1/2}U)$ is defined as follows: for $\xi\in L^2(M)$,
\[
R_U(k)\xi(x)=\int_{t^{-1}(x)} (\rho^U\cdot k)(u) \xi(s(u)), \quad
x\in M.
\]
\end{definition}

First, observe that two longitudinal kernels associated with the different
bi-submersions can define the same operator in $L^2(M)$.
Let $\varphi : M\to N$ be a submersion, and let $E$ be a vector
bundle on $N$. Integration along the fibers gives rise to a linear
map $\varphi_{!} : C_c(M, \Omega^1(\ker d\varphi)\otimes
\varphi^*E)\to C_c(N,E)$ defined by
\[
\varphi_{!}(f)(x)=\int_{\varphi^{-1}(x)}f, \quad x\in N.
\]
As shown in \cite{Andr-Skandalis-I}, if $\varphi : U\to V$ is a
morphism of bi-submersions which is a submersion, then for every
$k\in C^\infty_c(U, \Omega^{1/2}U)$, we have
$R_U(k)=R_V(\varphi_{!}(k))$. More generally, let $k_1\in
C^\infty_c(U_1, \Omega^{1/2}U_1)$ and $k_2\in C^\infty_c(U_2,
\Omega^{1/2}U_2)$. Assume that there exists a submersion $p : W\to
U_1\sqcup U_2$, which is a morphism of bi-submersions, and $k\in
C^\infty_c(W, \Omega^{1/2}W)$ such that $p_!(k)=(k_1,k_2)$. Moreover, suppose that there exists a morphism $q: W\to V$ of bi-submersions,
which is a submersion, such that $q_!(k)=0$. Then
$R_{U_1}(k_1)=R_{U_2}(k_2)$.

To describe the composition of regularizing operators, we first
recall \cite[p. 24]{Andr-Skandalis-I} that, for any
bi-submersions $(U_j,t_j,s_j)$, $j=1,2$, there exists a canonical
isomorphism
\[
\Omega^{1/2}(U_1\circ U_2)_{(u_1,u_2)}\cong
\Omega^{1/2}(U_1)_{u_1}\otimes \Omega^{1/2}(U_2)_{u_2}.
\]

\begin{propos}{\cite[p. 32]{Andr-Skandalis-I}}
(1) For any $k_1\in C^\infty_c(U_1, \Omega^{1/2}U_1)$ and $k_2\in
C^\infty_c(U_2, \Omega^{1/2}U_2)$, we have
\[
R_{U_1}(k_1)\circ R_{U_2}(k_2)=R_{U_1\circ U_2}(k_1\otimes k_2),
\]
where $
k_1\otimes k_2\in C^\infty_c(U_1, \Omega^{1/2}U_1)\otimes
C^\infty_c(U_2, \Omega^{1/2}U_2)\cong C^\infty_c(U_1\times U_2,
\Omega^{1/2}(U_1\circ U_2))$.

(2) For any $k\in C^\infty_c(U, \Omega^{1/2}U)$, we have
\[
R_{U}(k)^*=R_{U^{-1}}(k^*),
\]
where $k^*=\bar k$ via the canonical isomorphism
$\Omega^{1/2}U^{-1}\cong \Omega^{1/2}U$.
\end{propos}

\subsection{Pseudodifferential operators}\label{s:psi}

In this section, we introduce the classes of pseudodifferential
operators on $M$ associated with the singular foliation $\mathcal F$
and describe their properties, following \cite{{Andr-Skandalis-II}}.
We will keep notation of the previous subsection.

Let $(U,t,s)$ be a bi-submersion and $V\subset U$ the identity
bisection. (Remark that $V$ may be empty.) Let $p: N\to V$ be the
normal bundle of the inclusion $V\hookrightarrow U$, that is,
$N_v=T_vU/T_vV$, $v\in V$. Choose a tubular neighborhood
$(U_1,\phi)$ of $V$ in $U$. Thus, $U_1$ is a neighborhood of $V$ in
$U$ and $\phi : U_1\to N$ is a local diffeomorphism such that
$\phi(v)=(v,0)$ for $v\in V$, and $d\phi\left|_V \right. : T_vU\to
T_{(v,0)}N$ induces the identity isomorphism $N_v=T_vU/T_vV\to
T_{(v,0)}N\cong N_v$. Let $p_{N^*} : N^*\to V$ be the conormal
bundle. Denote $ N^*U_1=\{(u,\eta)\in U_1\times N^* :
p(\phi(u))=p_{N^*}(\eta)\}$.

The space $\mathcal P^m_c(U,V;\Omega^{1/2})$ of pseudodifferential
kernels of order $m$ consists of all $k\in C^{-\infty}_c(U,
\Omega^{1/2}U)=C^\infty(U, \Omega^1(TU)\otimes
\Omega^{-1/2}U)^\prime$ of the form
\begin{multline*}
\langle k,f\rangle =\int_U k_0(u)f(u)+(2\pi)^{-d}\int_{N^*U_1} e^{-i\langle \phi(u),\eta\rangle} \chi(u)\cdot a(p\circ \phi(u),\eta) f(u), \\
f\in C^\infty_c(U, \Omega^1(TU)\otimes \Omega^{-1/2}U),
\end{multline*}
where $k_0\in C^\infty(U,  \Omega^{1/2}U)$, $d={\rm rank}\, N=\dim
U-\dim M$, $\chi\in C^\infty_c(U)$ is such that ${\rm supp}\,\chi
\subset U_1$ and $\chi\equiv 1$ in a neighborhood of $V$, $a\in
S^m_{cl}(V, N^*; \Omega^1N^*\otimes \Omega^{1/2}U\left|_V\right.)$.

Here, for any $v\in V$, the section $a(v,\cdot)$ is a smooth density on $N^*_v$
with values in the vector space $\Omega^{1/2}U_v$. In a local
coordinate system on an open set $V_1\subset {\mathbb R}^n \cong
V_1\subset V$  and a trivialization of the vector bundle $N^*$ over
it, it is written as $a(v,\eta)|d\eta|$, $v\in V_1$, $\eta\in
{\mathbb R}^d$, where $a\in S^m(V\times {\mathbb R}^d;
\Omega^{1/2}U)$ is a classical symbol of order $m$.

Note that, in \cite{Andr-Skandalis-II}, the authors assume that the
order $m$ is integer, but it is easy to see that all the results of \cite{Andr-Skandalis-II} can be easily extended to the case of an arbitrary real $m$.

\begin{remark}
In \cite{Andr-Skandalis-II}, elements of $\mathcal
P^m_c(U,V;\Omega^{1/2})$ are called generalized sections of the
bundle $\Omega^{1/2}U$ with compact support and pseudodifferential
singularities along $V$ of order $\leq m$. In fact, they are just
conormal distributions on $U$ for the submanifold $V\subset U$ (see, for instance, \cite{HormanderIII}).
\end{remark}

With any pseudodifferential kernel $k\in \mathcal
P^m_c(U,V;\Omega^{1/2})$, we associate an operator $R_U(k) :
C^\infty_c(s(U))\to C^\infty_c(t(U))$ as follows: for $f\in
C^\infty_c(s(U))$, we put
\begin{multline*}
R_U(k)f(x)=R_U(k_0)f(x)\\
+(2\pi)^{-d}\int_{N^*U^x_1} e^{-i\langle \phi(u),\xi\rangle}
\chi(u)\rho^U_u\cdot [a(p\circ \phi(u),\xi)]f(s_U(u)).
\end{multline*}
Here  $N^*U^x_1=\{(u,\eta)\in N^*U_1 : t(u)=x\}$.

Remark that, if $V$ is empty, then $R_U(k)$ is a regularizing operator. Using an
appropriate cut-off function, the operator $R_U(k)$ can be uniquely
extended to an operator $R(k)$ on $C^\infty(M)$.

Observe that the bundle $\Omega^1N^*\otimes
\Omega^{1/2}U\left|_V\right.$ is canonically trivial. Indeed, since
$V\subset U$ is an identity bisection of the bi-submersion $(U,t,s)$,
by definition, the restriction $s\left|_V\right. : V\to M$ of $s$ to
$V$ is a diffeomorphism to an open subset $s(V)$. It follows that
$d(s\left|_V\right.)_v=ds_v\left|_{T_vV}\right. :
T_vV\stackrel{\cong}\to T_{s(v)}M$. On the other hand, we have a
short exact sequence $0\to \ker ds_v \to T_vU \to T_{s(v)}M \to 0$,
which implies that $\ker ds_v\cong T_vU/T_{s(v)}M \cong
T_vU/T_vV=N_v$. Similarly, we get an isomorphism $\ker dt_v\cong
N_v$. Therefore, we have
\[
\Omega^1N^*\otimes \Omega^{1/2}U\left|_V\right. \cong
\Omega^1N^*\otimes \Omega^{1/2}(\ker ds)\left|_V\right.\otimes
\Omega^{1/2}(\ker dt)\left|_V\right. \cong V\times \mathbb C.
\]
Thus, we can consider the (full) symbol $a$ of the operator $R_U(k)$ as an element of $S^m_{cl}(V, N^*)$. The principal symbol $\tilde
\sigma_m(R_U(k))$ of $R_U(k)$ is defined as the homogeneous component of degree $m$ of $a$:
\begin{equation}\label{e:prin}
\tilde  \sigma_m(R_U(k))(v,\xi)=a_m(v,\xi), \quad v\in V,\quad
\xi\in N^*_v\setminus \{0\}.
\end{equation}
So $\tilde  \sigma_m(R_U(k))$ is a smooth, degree $m$ homogeneous 
function on $N^*\setminus 0$.

\begin{definition}
The class $\Psi^m({\mathcal F})$ consists of operators $P$ in
$C^\infty(M)$ of the form $P=\sum_{i=1}^dP_i$. where each operator $P_i$, $i=1,\ldots, d$, has the form $P_i=R(k_i)$ and $k_i \in \mathcal P_c^m(U_i, V_i, \Omega^{1/2})$ for some bi-submersion $(U_i,t_i,s_i)$ and identity bisection $V_i\subset
U_i$.
\end{definition}

In order to define the principal symbol of an operator from
$\Psi^m({\mathcal F})$, we first introduce the cotangent bundle of
$\mathcal F$ as $\mathcal F^*=\bigsqcup_{x\in M}\mathcal F^*_x$, where, for any $x\in M$, $\mathcal F^*_x$ is the dual space of $\mathcal F_x$, the fiber of $\mathcal F$ at $x$. Observe that $\mathcal F^*$ is not a vector bundle in the usual sense. One
can show that $\mathcal F^*$ is a locally compact topological space.

Let $(U,t,s)$ be a bi-submersion and $V\subset U$ the identity
bisection. Recall that $N\cong (\ker ds)\left|_V\right.\cong (\ker
dt)\left|_V\right.$. Therefore, by Definition~\ref{d:bisub}, for
$v\in V$, $ds_v : N_v\to \mathcal F_x$, $x=s(v)$, is an epimorphism.
So the dual map $ds^*_v$ embeds $\mathcal F^*_x$ to
$N^*_v$. The longitudinal principal symbol of the operator $R_U(k)$
associated with $k\in \mathcal P^m_c(U,V;\Omega^{1/2})$ is the homogeneous function $\sigma_m(R_U(k))$ of degree $m$ on $\mathcal F^*\setminus 0$, which is equal to zero on $\mathcal F^*_x\setminus \{0\}$ for $x\not \in s(V)$ and for $x\in s(V)$ is defined on $\mathcal F^*_x\setminus \{0\}$ by
\begin{equation}\label{e:symbol}
\sigma_m(R_U(k))(x,\xi)=\tilde
\sigma_m(R_U(k))(v,ds_{v}^*(\xi)),\quad \xi \in
\mathcal F^*_x\setminus \{0\},
\end{equation}
where $v=s^{-1}(x)$ and $\tilde  \sigma_m(R_U(k))\in C^\infty(N^*\setminus 0)$ is the (local) principal symbol of $R_U(k)$ defined by \eqref{e:prin}.

Extending by linearity the principal symbol map to $\Psi^m({\mathcal F})$, we get the longitudinal principal symbol map
$\sigma_m : \Psi^m({\mathcal F})\to C(\mathcal F^*\setminus 0)$.
One can show that this map is well-defined.

\begin{theorem}[cf. Theorem 3.15 in \cite{Andr-Skandalis-II}]\label{t:comp}
Given $P_i\in \Psi^{m_i}({\mathcal F})$, $i=1,2$, their composition
$P=P_1\circ P_2$ is in $\Psi^{m_1+m_2}({\mathcal F})$ and
$\sigma_{m_1+m_2}(P)=\sigma_{m_1}(P_1)\sigma_{m_2}(P_2)$.
\end{theorem}

\begin{theorem}
Given $P_i\in \Psi^{m_i}({\mathcal F})$, $i=1,2$, the commutator
$[P_1, P_2]$ is in $\Psi^{m_1+m_2-1}({\mathcal F})$.
\end{theorem}

\begin{proof}
The proof of this theorem follows the arguments of \cite[Theorem
3.15]{Andr-Skandalis-II}, using a slight modification of the proof
of \cite[Proposition 1.10]{Andr-Skandalis-II}.
\end{proof}

An operator $P\in \Psi^{m}({\mathcal F})$ is said to be
longitudinally elliptic, if its longitudinal principal symbol
$\sigma_{m}(P)$ is invertible.

\begin{theorem}[cf. Theorem 4.2 in \cite{Andr-Skandalis-II}]\label{t:reg}
Given  a longitudinally elliptic operator $P\in \Psi^{m}({\mathcal
F})$, there is an operator $Q\in \Psi^{-m}({\mathcal F})$ such that
$1-P\circ Q$ and $1-Q\circ P$ are in $\Psi^{-\infty}({\mathcal F})$.
\end{theorem}

\begin{theorem}[cf. Theorem 5.3 in \cite{Andr-Skandalis-II}]\label{t:bounded}
Any operator $P\in \Psi^{0}({\mathcal F})$ defines a bounded
operator in $L^2(M)$.
\end{theorem}

\subsection{Examples}\label{s:examples}
1. Suppose that $\mathcal F$ is a smooth foliation on a
compact manifold $M$. Then one can define a bi-submersion $(U,t,s)$ as follows. $U=G$ is the holonomy groupoid of $\mathcal F$ (assume that it
is Hausdorff) and $t, s : G\to M$ are the usual target and source
maps of $G$. (We refer the reader to \cite{survey,umn-survey} for basic notions of noncommutative geometry of foliations.) An identity bisection $V$ of this bi-submersion is given by the unit set $G^{(0)}\subset G$ of the groupoid $G$. The bundle $\Omega^{1/2}U$ is
the leafwise half-density bundle associated with a natural
$2p$-dimensional foliation ${\mathcal G}$ on $G$, and the space $C^\infty_c(U,\Omega^{1/2}U)$ is a basic element for constructing operator algebras associated with $\mathcal F$. Finally, 
the space $\mathcal P^m_c(U,V;\Omega^{1/2})$ coincides with the space of kernels of $G$-pseudodifferential operators introduced in \cite{Connes79}. 

2. As above, suppose that $\mathcal F$ is a smooth foliation on a
compact manifold $M$. Let $\phi: D\stackrel{\cong}{\to} I^p\times I^q$ and $\phi': D'\stackrel{\cong}{\to} I^p\times I^q$ be two compatible foliated charts on $M$ (here
$I=(0,1)$) and $W(\phi,\phi')\subset G \stackrel{\cong}{\to} I^p\times I^p\times I^q$ the corresponding coordinate chart on the holonomy groupoid $G$ \cite{Connes79} (see also \cite{survey,umn-survey}). Then we have 
a bi-submersion $(U,t,s)$, where $U=W(\phi,\phi')$ and $t : W(\phi,\phi') \to D$ and $s : W(\phi,\phi') \to D^\prime$ are the restrictions of the target and source maps of the holonomy groupoid $G$ to $W(\phi,\phi')$. In local coordinates, they are given by
\[
t(x,x^\prime,y)=(x,y),\quad s(x,x^\prime,y)=(x^\prime,y), \quad (x,x^\prime,y)\in I^p\times I^p\times I^q.
\]
In the charts $\phi$ and $\phi^\prime$, the positive smooth density $\mu$ can be written as $\mu=\mu(x,y)|dx||dy|$ and $\mu=\mu^\prime(x^\prime,y^\prime)|dx^\prime||dy^\prime|$, respectively. There are natural sections of the bundles $\Omega^{1/2}\ker ds$ and $\Omega^{1/2}\ker dt$, which can be written as $|dx|^{1/2}$ and $|dx^\prime|^{1/2}$, respectively. Then $\rho^U\in C^\infty(U,\Omega^{-1/2}\ker ds\otimes \Omega^{1/2}\ker dt)$ is given by
\[
\rho^U_{(x,x^\prime,y)}=\left(\frac{\mu^\prime(x^\prime,y)}{\mu(x,y)}\right)^{1/2}|dx|^{-1/2} |dx^\prime|^{1/2}, \quad (x,x^\prime,y)\in I^p\times I^p\times I^q.
\]
Any $k\in C^\infty_c(U,\Omega^{1/2}U)$ has the form
$k=K(x,x^\prime,y)|dx|^{1/2}|dx^\prime|^{1/2}$ with $K\in C^\infty_c(I^p\times I^p\times I^q)$, and the operator
$R_U(k) : C^\infty(D^\prime)\to C^\infty(D)$ is given
by
\[
R_U(k)f(x,y)=\int K(x,x^\prime,y)\left(\frac{\mu^\prime(x^\prime,y)}{\mu(x,y)}\right)^{1/2}
f(x^\prime,y)dx^\prime.
\]
In the case when $\phi=\phi^\prime$, a non-empty identity bisection $V\subset W(\phi,\phi)\cong I^p\times I^p\times I^q$ is given by
\[
V=\{(x,x^\prime,y) \in I^p\times I^p\times I^q : x=x^\prime\}\cong
I^p\times I^q\cong D.
\]
Then we have $N\cong I^p\times I^q\times {\mathbb R}^p$ and a diffeomorphism $\phi : U_1\subset U \to N$ can be taken in the form
\[
\phi : (x,x^\prime,y)\in I^p\times I^p\times I^q \mapsto (x,y,x^\prime-x)\in I^p\times I^q\times {\mathbb R}^p.
\]
Finally, a symbol $a\in S^m_{cl}(V,N^*)$ is written as $a=a(x,y,\xi)$, $(x,y,\xi)\in I^p\times I^q\times ({\mathbb R}^p)^*$, and the corresponding operator $P : C^\infty_c(D^\prime)\to C^\infty_c(D)$ is given by, for $f\in
C^\infty_c(D^\prime)\cong C^\infty_c(I^p\times I^q)$,
\begin{multline*}
Pf(x,y)\\ = (2\pi)^{-p}\int_{I^p}\int_{\mathbb R^p} e^{i\langle x-x^\prime,\xi\rangle}
\chi(x,x^\prime,y) a(x,y,\xi)f(x^\prime,y)\left(\frac{\mu^\prime(x^\prime,y)}{\mu(x,y)}\right)^{1/2}
|dx^\prime| |d\xi|.
\end{multline*}

3. Suppose that $\mathcal F$ is a singular foliation on a compact manifold $M$. We show that any vector field $X\in \mathcal F$ considered as a first order differential operator on $M$ belongs to $\Psi^1(\mathcal F)$, and its principal symbol $\sigma_1(X) \in C(\mathcal F^*\setminus 0)$ is
given by
\[
\sigma_1(X)(\xi)=i\langle \xi, X\rangle, \quad \xi \in \mathcal F^*.
\]

First, we consider an arbitrary bi-submersion $(U,t,s)$ and a nonempty identity
bisection $V\subset U$ and assume that $X\in \mathcal F$ is
supported in $s(V)$. Since $s:U\to M$ is a submersion, there exists
a vector field $\tilde X\in C^\infty (U, TU)$ such that $ds_u(\tilde
X(u))=X(s(u))$. Without loss of generality, we can assume that
$\tilde X\in C^\infty(U,\ker dt)$. Indeed, $\tilde X$ is defined up
to $C^\infty(U,\ker ds)$ and, by definition, $\tilde X\in
s^{-1}(\mathcal F)=C^\infty(U,\ker ds)+C^\infty(U,\ker dt)$.  Then
the restriction of $\tilde X$ to $V$ belongs to $C^\infty(V,(\ker
dt)\left|_V\right.)\cong C^\infty(V,N)$, giving rise to a vector
field $\tilde X\in C^\infty(V,N)$.

Put
\[
a_X(v,\xi)=i\langle \xi, \tilde X(v)\rangle, \quad v\in V,\quad
\xi\in N^*_v.
\]
The corresponding pseudodifferential operator $P :
C^\infty_c(s(U))\to C^\infty_c(t(U))$ is given for $f\in
C^\infty_c(s(U))$ by
\begin{multline*}
Pf(x) =(2\pi)^{-d}\int_{U^x}  \int_{N^*_{p\circ \phi(u)}}
e^{-i\langle \phi(u),\xi\rangle} \chi(u) i \langle \xi, \tilde
X(p\circ \phi(u))\rangle\times \\ \times f(s(u))\rho^U\cdot {\bf 1}, \quad x\in t(U).
\end{multline*}
Here $\bf 1$ is a smooth section of the bundle $\Omega^1N^*\otimes
\Omega^{1/2}U$, corresponding to its canonical trivialization.
Thus, $\rho^U\cdot {\bf 1}$ is a smooth section of $\Omega^1N^*\otimes
\Omega^{1}\ker dt$, which can be integrated over $N^*U_1$. 

Since $\tilde X\in C^\infty(U,\ker dt)$, it is tangent to $U^x$, and
we can use the formula
\[
\int_{U^x}(\tilde XF)\omega=-\int_{U^x}{\rm div}_\omega(\tilde
X)F\omega,
\]
that holds for any function $F\in C^\infty_c(U)$ and for any smooth
positive section $\omega\in C^\infty (U,\Omega^1\ker
dt)$. By this formula, we obtain that
\[
Pf(x)=(2\pi)^{-d}\int_{U^x}  \int_{N^*_{p\circ \phi(u)}}
e^{-i\langle \phi(u),\xi\rangle}(\tilde X+w)[\chi(u)
f(s(u))] \rho^U\cdot {\bf 1}
\]
with some $w\in C^\infty(U)$. Now, using the Fourier
transform inversion formula and observing that $\phi(u)=0
\Leftrightarrow u\in V \Leftrightarrow s_U(u)=x$, we get
\[
Pf(x)=(\tilde
X+w)[f(s(u))\left|_{u=s^{-1}(x)}\right.=Xf(x)+w(s^{-1}(x))f(x).
\]
Thus, we conclude that $X=P-(s^{-1})^*w$ has a kernel in $\mathcal
P^1(U,V,\Omega^{1/2})$.

Using \eqref{e:symbol}, we compute its longitudinal principal symbol:
\[
\sigma_1(X)(x,\xi)=a_X(v,ds_v^*(\xi))=i\langle \xi, ds_v(\tilde
X(v))\rangle=i\langle \xi, X(x)\rangle, \quad \xi \in \mathcal
F^*_x, s(v)=x.
\]

Now take a finite family $(U_\alpha,t_\alpha,s_\alpha),
\alpha=1,\ldots,d,$ of bi-submersions equipped with identity
bisections $V_\alpha\subset U_\alpha$ such that $M=\cup_{\alpha=1}^d
s(V_\alpha)$, a partition of unity $\phi_\alpha\in C^\infty(M), \alpha=1,\ldots,d,$ subordinate to the covering $\{s(V_\alpha)\}$, ${\rm supp}\,\phi_\alpha\subset s(V_\alpha)$, and a family of smooth functions $\psi_\alpha\in C^\infty(M),
\alpha=1,\ldots,d,$ such that ${\rm supp}\,\psi_\alpha\subset s(V_\alpha)$, $\phi_\alpha\psi_\alpha=\phi_\alpha$. Then we write
$X=\sum_{\alpha=1}^d \phi_\alpha X \psi_\alpha$ to see that $X$
belongs to $\Psi^1(\mathcal F)$.

\subsection{Longitudinal Sobolev spaces}

First, we observe that, for any $s\in \mathbb R$, there exists a
longitudinally elliptic operator $\Lambda_s$ of order $s$. To
construct such an operator, first we take, as above, a finite family
$(U_\alpha,t_\alpha,s_\alpha), \alpha=1,\ldots,d,$ of bi-submersions
equipped with identity bisections $V_\alpha\subset U_\alpha$ such
that $M=\cup_{\alpha=1}^d s(V_\alpha)$, a partition of unity
$\phi_\alpha\in C^\infty(M)$ subordinate to the covering of $M$,
${\rm supp}\,\phi_\alpha\subset s(V_\alpha)$, and $\psi_\alpha\in
C^\infty(M)$ such that ${\rm supp}\,\psi_\alpha\subset s(V_\alpha)$,
$\phi_\alpha\psi_\alpha=\phi_\alpha$. Then, for each $\alpha$, we
consider an operator $P_\alpha$ defined by a pseudodifferential
kernel $k_\alpha\in \mathcal P^s_c(U_\alpha,V_\alpha;\Omega^{1/2})$
with the symbol $a(x,\xi)=(1+|\xi|)^s$. Finally, we put
$\Lambda_s=\sum_{\alpha=1}^d \phi_\alpha P_\alpha \psi_\alpha$.

We fix such an operator $\Lambda_s$ for any $s$. Without loss of
generality, we can assume that $\Lambda_s$ is formally self-adjoint and
\[
\Lambda_s\circ \Lambda_{-s}=I+R_s, \quad \Lambda_{-s}\circ
\Lambda_{s}=I+R^\prime_s, \quad R_s, R^\prime_s\in \Psi^{-\infty}(\mathcal F).
\]

\begin{definition}
For $s\geq 0$, the Sobolev space $H^s(\mathcal F)$ is defined as the
domain of $\Lambda_s$ in $L^2(M)$:
\[
H^s(\mathcal F)=\{u\in L^2(M) : \Lambda_su\in L^2(M)\}.
\]
 The norm in $H^s(\mathcal F)$ is
defined by the formula
\[
\|u\|^2_{s}=\|\Lambda_su\|^2+\|u\|^2, \quad u\in H^s(\mathcal F).
\]

For $s<0$, $H^s(\mathcal F)$ is defined as the dual space of
$H^{-s}(\mathcal F)$.
\end{definition}

Using Theorems~\ref{t:comp} and~\ref{t:bounded}, we immediately get the following result.

\begin{theorem}\label{t:action_in_Sobolev}
For any $s\in \mathbb R$, an operator $A\in \Psi^m(\mathcal F)$
determines a bounded operator $A : H^s(\mathcal F)\to H^{s-m}(\mathcal F)$.
\end{theorem}

\begin{propos}\label{p:density_in_Sobolev}
For $s\in \mathbb Z$, the space $C^\infty(M)$ is dense in
$H^s(\mathcal F)$.
\end{propos}

\begin{proof}
The proof can be easily given, using the standard Friedrichs' mollifiers on $M$ (see, for instance, \cite[Chapter II, \S 7]{Taylor}).
\end{proof}

\begin{theorem}
Any formally self-adjoint longitudinally elliptic operator $P\in
\Psi^{m}({\mathcal F})$, $m>0$, defines an unbounded self-adjoint
operator in the Hilbert space $L^2(M,\mu)$ with the domain $H^m(\mathcal F)$.
\end{theorem}

Let us note that the results of this subsection can be obtained as consequences of the general results of \cite[Section 6]{Andr-Skandalis-II} applied to the natural representation of the $C^*$-algebra $C^*(M,\mathcal F)$ of the singular foliation $\mathcal F$ on $L^2(M,\mu)$. In particular, the Sobolev space $H^k(\mathcal F)$ is obtained as the image  of $L^2(M,\mu)$ by the action of the Sobolev module $H^k\subset C^*(M,\mathcal F)$ under this representation.

\section{Proofs of Theorems~\ref{t:Hs-hypo} and \ref{t:hypo}}\label{s:subelliptic}

The proof of Theorem~\ref{t:Hs-hypo} will closely follow Kohn's proof of the
subellipticity of the H{\"o}rmander's operators \cite{Kohn} (see
also \cite{Treves,Helffer-Nier}). We will keep notation of
Section~\ref{s:proofs}. The starting point is the following fact.

\begin{lemma}
For any $X\in C^\infty(M,H)$, there exists $C>0$ such that
\begin{equation}\label{Xestimate}
\|Xu\|^2 \leq C\left((\Delta_Hu,u)+\|u\|^2\right), \quad
u\in C^\infty(M).
\end{equation}
\end{lemma}

\begin{proof}
Let $\Omega$ be an open subset of $M$ such that there exists a local
orthonormal frame $X_1,\ldots, X_p\in C^\infty(\bar{\Omega},
H\left|_{\bar\Omega}\right.)$. Then, for any $u\in
C^\infty_c(\Omega)$, we have
\[
(\Delta_Hu,u)=\|d_Hu\|_g^2=\sum_{j=1}^p\int_\Omega|X_ju(x)|^2\,d\mu(x).
\]
We can write $X(x)=\sum_{j=1}^pa_j(x)X_j(x), x\in \bar\Omega$ with some $a_j\in C^\infty(\bar\Omega), j=1,\ldots,p$. Therefore, for any $u\in C^\infty_c(\Omega)$, we get
\[
\|Xu\|^2=\int_\Omega |Xu(x)|^2\,d\mu(x) \leq C\sum_{j=1}^p\int_\Omega |X_ju(x)|^2\,d\mu(x) = C(\Delta_Hu,u).
\]

To prove the estimate \eqref{Xestimate} in the general case, we take a
finite open covering $M=\cup_{\alpha=1}^d \Omega_\alpha$ of $M$ such that, for any $\alpha=1,\ldots,d$, there exists
a local orthonormal frame $X^{(\alpha)}_1,\ldots, X^{(\alpha)}_p\in
C^\infty(\Omega_\alpha, H\left|_{\Omega_\alpha}\right.)$ and a
partition of unity subordinate to this covering, and use the fact
that, for any $\varphi\in C^\infty(M)$, the commutators
$[X,\varphi]$ and $[d_H,\varphi]$ are zero order differential
operators and, therefore, bounded in $L^2$.
\end{proof}

We start the proof of Theorem~\ref{t:Hs-hypo} with the case $s=0$.

\begin{propos}\label{t:L2-hypo}
There exist $\epsilon>0$ and $C>0$ such that
\[
\|u\|_\epsilon^2 \leq C\left(\|\Delta_Hu\|^2+\|u\|^2\right), \quad
u\in C^\infty(M).
\]
\end{propos}

\begin{proof}
Let $\mathcal P$ be the set of all operators $P\in \Psi^0(\mathcal
F)$ such that there exist constants $\epsilon>0$ and $C>0$ such that
\begin{equation}\label{mathcalP}
\|Pu\|_\epsilon^2 \leq C\left(\|\Delta_Hu\|^2+\|u\|^2\right), \quad u\in C^\infty(M).
\end{equation}

We claim that $\mathcal P$ satisfies the following properties:

\begin{description}
\item[(P0)]  $\bigcup_{m<0}\Psi^m(\mathcal F)$ is in $\mathcal P$;
\item[(P1)]  $\mathcal P$ is a two-sided ideal in $\Psi^0(\mathcal F)$;
\item[(P2)]  $\mathcal P$ is stable by taking the adjoints;
\item[(P3)] $X\Lambda_{-1}\in \mathcal P$ for $X\in C^\infty(M,H)$;
\item[(P4)]  If $P\in \mathcal P$ then $[X,P]\in \mathcal P$ for $X\in C^\infty(M,H)$.
\end{description}

\begin{proof}[Proof of (P2)]
First, observe that
\[
\|\Lambda_\epsilon P^*u\|^2=\langle P \Lambda^2_\epsilon P^*u, u\rangle=\|\Lambda_\epsilon Pu\|^2 + \langle (P \Lambda^2_\epsilon P^*-P^* \Lambda^2_\epsilon P)u, u\rangle.
\]
It remains to note that $P \Lambda^2_\epsilon P^*-P^* \Lambda^2_\epsilon P\in \Psi^0(\mathcal F)$ if $\epsilon<\frac 12$.
\end{proof}

\begin{proof}[Proof of (P1)]
First, observe that, by Theorem~\ref{t:action_in_Sobolev}, $\mathcal P$ is a left ideal. Then, by (P2), it is a right ideal as well.
\end{proof}

\begin{proof}[Proof of (P3)]
Using \eqref{Xestimate}, we have
\[
\|\Lambda_{-1} Xu\|_1^2\leq C \|X u\|^2\leq C_1 (\|\Delta_H u\|^2+\|u\|^2),
 \]
which means that $\Lambda_{-1} X\in \mathcal P$. Therefore, it follows from (P2) that $(\Lambda_{-1} X)^*=X^*\Lambda_{-1}\in \mathcal P$.

Since $X^*=-X+c$ with some $c\in C^\infty(M)$, using (P0) and (P1), we get
\[
X^*\Lambda_{-1}=-X\Lambda_{-1}+c\Lambda_{-1}\in \mathcal P.\qedhere
\]
\end{proof}

\begin{proof}[Proof of (P4)]
Take $P\in \Psi^0(\mathcal F)$ such that $P$ and $P^*$  satisfy
\eqref{mathcalP} with some $\epsilon>0$. For $\delta>0$, one can
write
\begin{multline}\label{e:10}
\|[X,P]u\|^2_\delta = ( [X,P]u, \Lambda^2_\delta [X,P]u)+\|[X,P]u\|^2\\
= (XPu, T_{2\delta}u)-(PXu, T_{2\delta}u)+\|[X,P]u\|^2,
 \end{multline}
 where $T_{2\delta}= \Lambda^2_\delta [X,P] \in  \Psi^{2\delta}(\mathcal F)$. For the second term in the right hand side of \eqref{e:10}, we have 
\begin{align*}
|(PXu, T_{2\delta}u)|  & = |(Xu, P^*T_{2\delta}u)|
\leq  \frac12(\|Xu\|^2+\|P^*T_{2\delta}u\|^2)\\
& \leq  \frac12 \|Xu\|^2+\|T_{2\delta}P^*u\|^2+\|[P^*,T_{2\delta}]u\|^2.
 \end{align*}
Assuming $\delta <\min (\frac 12, \frac{\epsilon}{2})$, we obtain that
\[
\|T_{2\delta}P^*u\|^2\leq C\|P^*u\|_{2\delta}^2\leq C\|P^*u\|_\epsilon^2 \leq C_1\left(\|\Delta_Hu\|^2+\|u\|^2\right)
\]
and
\[
\|[P^*,T_{2\delta}]u\|^2\leq C\|u\|^2,
\]
which proves the estimate
\begin{equation}\label{e:11}
|(PXu, T_{2\delta}u)| \leq C_1\left(\|\Delta_Hu\|^2+\|u\|^2\right).
\end{equation}
Similarly,  for the first term in the right hand side of \eqref{e:10}, we have 
\begin{multline*}
|(XPu, T_{2\delta}u)|  = |(Pu, X^*T_{2\delta}u)|\leq   |(Pu, XT_{2\delta}u)|+ |(Pu, cT_{2\delta}u)|\\
 \leq  |(Pu,T_{2\delta}Xu)|+|(Pu, [X,T_{2\delta}]u)| + |(Pu, cT_{2\delta}u)|.
 \end{multline*}
Now we proceed as follows, using $\delta < \frac{\epsilon}{2}$ and \eqref{Xestimate}:
\[
 |(Pu,T_{2\delta}Xu)|= |(T_{2\delta}^*Pu,Xu)|\leq C\|Pu\|_{2\delta}\|Xu\|\leq C_1\left(\|\Delta_Hu\|^2+\|u\|^2\right),
 \]
 \[
 |(Pu, cT_{2\delta}u)|=|(T_{2\delta}^*c^*Pu, u)|\leq C\|Pu\|_{2\delta}\|u\|\leq C_1\left(\|\Delta_Hu\|^2+\|u\|^2\right),
 \]
 and, finally,
\[
|(Pu, [X,T_{2\delta}]u)|=|([X,T_{2\delta}]^*Pu, u)|\leq C\|Pu\|_{2\delta}\|u\|\leq C_1\left(\|\Delta_Hu\|^2+\|u\|^2\right).
\]
We obtain that
\begin{equation}\label{e:12}
|(XPu, T_{2\delta}u)| \leq C_2\left(\|\Delta_Hu\|^2+\|u\|^2\right).
\end{equation}
Plugging \eqref{e:11} and \eqref{e:12} into \eqref{e:10}, we complete the proof of (P4).
\end{proof}

Now we complete the proof of Proposition~\ref{t:L2-hypo}.
First, we claim that, for any $X_1,\ldots,X_p\in C^\infty(M,H)$, the operator $[X_1,[X_2,\ldots, [X_{p-1},X_p]\ldots ]]\Lambda_{-1}$ belongs to $\mathcal P$.  We proceed by induction. Let us write
\[
[X_1,[X_2,\ldots, [X_{p-1},X_p]\ldots ]]=[X_1,Y], \quad Y=[X_2,\ldots, [X_{p-1},X_p]\ldots ],
\]
and assume that, by the induction hypothesis, $Y\Lambda_{-1}\in \mathcal P$. Then, by (P4), we know that $[X_{1},Y\Lambda_{-1}]\in \mathcal P$.
On the other hand, we can write
\[
[X_{1},Y\Lambda_{-1}]=[X_{1},Y]\Lambda_{-1}
+Y[X_{1},\Lambda_{-1}].
\]
Since $\Lambda_{-1}\Lambda_1=I+R_1$ with $R_1\in \Psi^{-\infty}(\mathcal F)$, we get
\[
Y[X_{1},\Lambda_{-1}]=Y\Lambda_{-1}\Lambda_1[X_{1},\Lambda_{-1}]-YR_1[X_{1},\Lambda_{-1}],
\]
that, by (P0) and (P1), immediately implies that $Y[X_{1},\Lambda_{-1}]\in \mathcal P$, since $Y\Lambda_{-1}\in \mathcal P$,  $\Lambda_1 [X_{1},\Lambda_{-1}]\in  \Psi^{0}(\mathcal F)$ and $YR_1[X_{1},\Lambda_{-1}]\in \Psi^{-\infty}(\mathcal F)$.
Thus we conclude that $[X_{1},Y]\Lambda_{-1}$ belongs to $\mathcal P$, that completes the proof.

By assumption, the $C^\infty(M)$-module $\mathcal F$ is generated by a finite set of vector fields $Y_1,\ldots,Y_N$ on $M$. Consider the operator $\Delta=\sum_{k=1}^NY^*_kY_k$, a Laplacian associated with $\mathcal F$. It is a formally self-adjoint, longitudinally elliptic, second order differential operator.
Let $Q\in \Psi^{-2}(\mathcal F)$ be its parametrix, i.e. $Q\Delta=I-K_1$, $\Delta Q=I-K_2$, $K_i\in \Psi^{-\infty}(\mathcal F)$. Then we have
\[
I=\sum_{j=1}^NQY^*_jY_j+K_1.
\]
Since $QY^*_j\in  \Psi^{-1}(\mathcal F)$, it follows from (P3) that $QY^*_jY_j\in \mathcal P$. By (P0), $K_1\in \mathcal P$. So we obtain that $I\in \mathcal P$, that completes the proof.
\end{proof}

Now we extend the subelliptic estimates of Proposition~\ref{t:L2-hypo} to an arbitrary $s$, completing the proof of Theorem~\ref{t:Hs-hypo}.

\begin{proof}[Proof of Theorem~\ref{t:Hs-hypo}]
By Proposition~\ref{t:L2-hypo}, we have
\[
\|u\|_{s+\epsilon}^2\leq c(\|\Lambda_s u\|_{\epsilon}^2+\|u\|_s^2) \leq C\left(\|\Delta_H \Lambda_s u\|^2+\|u\|_s^2\right).
\]
It remains to show that
\begin{equation}\label{e:DL}
\|\Delta_H \Lambda_s u\|^2\leq C^\prime_s\left(\|\Delta_Hu\|_s^2+\|u\|_s^2\right)
\end{equation}

\begin{lemma}\label{l:com}
The operator $[\Delta_H, \Lambda_s]$ can be represented in the form
\[
[\Delta_H, \Lambda_s]=\sum_{k=1}^N T^s_k X_k+T^s_0,
\]
where $X_k\in C^\infty(M,H), k=1,\ldots, N,$ and $T^s_k \in
\Psi^s(\mathcal F), k=0,\ldots, N$.
\end{lemma}

\begin{proof}
Let $M=\bigcup_{\alpha=1}^d \Omega_\alpha$ be a finite open covering of
$M$ such that, for any $\alpha=1,\ldots,d$, there exists a local
orthonormal frame $X^{(\alpha)}_1,\ldots, X^{(\alpha)}_p\in
C^\infty(\Omega_\alpha, H\left|_{\Omega_\alpha}\right.)$. As mentioned above, the restriction of $\Delta_H$ to $\Omega_\alpha$ is written as
\[
\Delta_H\left|_{\Omega_\alpha}\right.=\sum_{j=1}^p(X^{(\alpha)}_j)^* X^{(\alpha)}_j.
\]
Let $\phi_\alpha\in C^\infty(M)$ be a partition of unity subordinate
to the covering, ${\rm supp}\,\phi_\alpha\subset U_\alpha$, and
$\psi_\alpha\in C^\infty(M)$ such that ${\rm
supp}\,\psi_\alpha\subset U_\alpha$,
$\phi_\alpha\psi_\alpha=\phi_\alpha$. Then we have
\begin{multline*}
\Delta_H=\sum_{\alpha=1}^d \phi_\alpha (\Delta_H\left|_{\Omega_\alpha}\right.)\psi_\alpha=\sum_{\alpha=1}^d\sum_{j=1}^p \phi_\alpha (X^{(\alpha)}_j)^* X^{(\alpha)}_j\psi_\alpha\\ =\sum_{\alpha=1}^d \sum_{j=1}^p \phi_\alpha (X^{(\alpha)}_j)^* \psi_\alpha X^{(\alpha)}_j + \sum_{\alpha=1}^d \sum_{j=1}^p \phi_\alpha (X^{(\alpha)}_j)^* [X^{(\alpha)}_j,\psi_\alpha].
\end{multline*}
We can write
\begin{align*}
\phi_\alpha (X^{(\alpha)}_j)^*\psi_\alpha X^{(\alpha)}_j\Lambda_s=& \phi_\alpha (X^{(\alpha)}_j)^*\Lambda_s \psi_\alpha X^{(\alpha)}_j+\phi_\alpha (X^{(\alpha)}_j)^*[\psi_\alpha X^{(\alpha)}_j,\Lambda_s]\\=& \Lambda_s \phi_\alpha (X^{(\alpha)}_j)^*\psi_\alpha X^{(\alpha)}_j+[\phi_\alpha (X^{(\alpha)}_j)^*,\Lambda_s] \psi_\alpha X^{(\alpha)}_j\\ &+[\psi_\alpha X^{(\alpha)}_j,\Lambda_s]\phi_\alpha (X^{(\alpha)}_j)^*+ [\phi_\alpha (X^{(\alpha)}_j)^*,[\psi_\alpha X^{(\alpha)}_j,\Lambda_s]].
\end{align*}
Since $(X^{(\alpha)}_j)^*=-X^{(\alpha)}_j+c^{(\alpha)}_j$ with some $c^{(\alpha)}_j\in C^\infty(M)$, we get
\[
\Delta_H\Lambda_s=\Lambda_s \Delta_H+\sum_{\alpha=1}^d \sum_{j=1}^p T^{s,(\alpha)}_{1,j} \psi_\alpha X^{(\alpha)}_j+ \sum_{\alpha=1}^d \sum_{j=1}^pT^{s,(\alpha)}_{2,j} \phi_\alpha X^{(\alpha)}_j+T^s_0,
\]
where the operators
\begin{align*}
T^{s,(\alpha)}_{1,j}=&[\phi_\alpha (X^{(\alpha)}_j)^*,\Lambda_s], \quad
T^{s,(\alpha)}_{2,j}=-[\psi_\alpha X^{(\alpha)}_j,\Lambda_s], \\
T^s_0=& \sum_{\alpha=1}^d \sum_{j=1}^p \Big( [\psi_\alpha X^{(\alpha)}_j,\Lambda_s]\phi_\alpha c^{(\alpha)}_j+[\phi_\alpha (X^{(\alpha)}_j)^*,[\psi_\alpha X^{(\alpha)}_j,\Lambda_s]]\\ & +[\phi_\alpha (X^{(\alpha)}_j)^* [X^{(\alpha)}_j,\psi_\alpha],\Lambda_s]\Big)
\end{align*}
belong to $\Psi^s(\mathcal F)$. Setting $\{X_k, k=1,\ldots, N\}=\{\psi_\alpha X^{(\alpha)}_j, \phi_\alpha X^{(\alpha)}_j, \alpha=1,\ldots,d, j=1,\ldots,p\}$ with $N=2dp$, we complete the proof.
\end{proof}

By Lemma~\ref{l:com}, it follows that there exists $C>0$ such that
\begin{equation}\label{e:0}
\|\Delta_H\Lambda_su\|^2\leq C(\|\Delta_Hu\|_s^2+\sum_{k=1}^N
\|X_ku\|_s^2+\|u\|^2_s), \quad u\in C^\infty(M).
\end{equation}
For any $k$, we have
\begin{multline}\label{e:1}
\|X_ku\|_s^2= \|\Lambda_sX_ku\|^2+ \|X_ku\|^2\leq
\|X_k\Lambda_s u\|^2+\|[\Lambda_s,X_k] u\|^2+ \|X_ku\|^2\\ \leq
\|X_k \Lambda_s u\|^2+ (\Delta_Hu,u)+C\| u\|_s^2.
\end{multline}
Next, by \eqref{Xestimate}, it follows that
\begin{equation}\label{e:2}
\begin{aligned}
\|X_k\Lambda_s u\|^2\leq & C( (\Delta_H\Lambda_su,\Lambda_su)+ \|u\|^2_s)\\
=&C((\Delta_Hu,u)_s+([\Delta_H,\Lambda_s]u,\Lambda_su)+\|u\|^2_s)\\
= & C((\Delta_Hu,u)_s+((\sum_{k=1}^N T^s_k
X_k+T^s_0)u,\Lambda_su)+\|u\|^2_s)\\ \leq &
C_1(\|\Delta_Hu\|^2_s+\sum_{k=1}^N \|X_ku\|_s \|u\|_s+\|u\|^2_s)\\
\leq & \epsilon \sum_{k=1}^N \|X_ku\|^2_s
+C_2(\epsilon)(\|\Delta_Hu\|^2_s+\|u\|^2_s)
\end{aligned}
\end{equation}
for any $\epsilon>0$ with some $C_2(\epsilon)>0$. From \eqref{e:1} and \eqref{e:2}, we
immediately get
\begin{equation}\label{e:3}
\sum_{k=1}^N  \|X_ku\|_s^2\leq C(\|\Delta_Hu\|_s^2+\|u\|^2_s).
\end{equation}
Plugging \eqref{e:3} into \eqref{e:0}, we get \eqref{e:DL}.
\end{proof}

\begin{proof}[Proof of Theorem~\ref{t:hypo}]
Following the standard construction of Friedrichs' mollifiers (see, for instance, \cite[Chapter II, \S 7]{Taylor} or \cite[Chapter II, \S 4]{Treves}), one can construct a bounded family $J_\varepsilon, 0<\varepsilon \leq 1,$ of operators from $\Psi^{-\infty}(\mathcal F)$ such that $J_\varepsilon u\to u$ in $L^2(M)$ as $\varepsilon\to 0$ for any $u\in L^2(M)$ and, for any $A\in \Psi^m(\mathcal F)$, the commutators $[A,J_\varepsilon]\in \Psi^{-\infty}(\mathcal F), 0<\varepsilon \leq 1,$ form a bounded family of operators in $\Psi^{m-1}(\mathcal F)$. More precisely, we first construct such a family  locally. Let $(U,t,s)$ be a bi-submersion and $V\subset U$ the identity bisection. In notation of Section~\ref{s:psi}, take a function $\rho\in C^\infty(N^*)$ supported in a tubular neighborhood $\phi(U_1)$ in $N^*$ such that $\rho\left|_V\right.\equiv 1$. One can check that the operator family $J_\varepsilon, 0<\varepsilon \leq 1,$ where the operator $J_\varepsilon=R_U(k_\varepsilon)$ is defined by the pseudodifferential kernel $k_\varepsilon\in \mathcal P^{-\infty}_c(U,V; \Omega^{1/2})$ with $k_{\varepsilon,0}=0$ and $a_\varepsilon(v,\eta)=\rho(v,\varepsilon\eta), v\in V, \eta\in N_v$, satisfies the desired conditions. The globally defined operator family $J_\varepsilon\in \Psi^{-\infty}(\mathcal F), 0<\varepsilon \leq 1$ is obtained from such families constructed locally by the usual gluing procedure (see, for instance, Example 3 of Section~\ref{s:examples}).

As an easy consequence, one get that, for any $s\in\mathbb R$, $J_\varepsilon u\to u$ in $H^s(\mathcal F)$ as $\varepsilon\to 0$ for any $u\in H^s(\mathcal F)$ and, for any $A\in \Psi^m(\mathcal F)$ and $B\in \Psi^{m^\prime}(\mathcal F)$, the operators $[B,[A,J_\varepsilon]]\in \Psi^{-\infty}(\mathcal F), 0<\varepsilon \leq 1,$ form a bounded family of operators in $\Psi^{m+m^\prime-2}(\mathcal F)$. Then one can easily complete the proof of the theorem, proceeding, for instance, as in the proof of \cite[Chapter II, Lemma 5.3]{Treves}.
\end{proof}

\end{document}